\documentclass{article}
\usepackage{amssymb,amsmath,amsthm,graphicx,xcolor,tikz}
\usetikzlibrary{knots}
\usetikzlibrary{hobby}
\usetikzlibrary{arrows.meta}
\usetikzlibrary{arrows,decorations.markings}
\def\utr{\, \underline{\triangleright}\, }
\def\otr{\, \overline{\triangleright}\, }
\def\ud{\, \underline{\bullet}\, }
\def\od{\, \overline{\bullet}\, }

\newtheorem{theorem}{Theorem}

\newtheorem{corollary}[theorem]{Corollary}

\theoremstyle{definition}
\newtheorem{example}{Example}
\newtheorem{definition}{Definition}

\date{}

\title{\Large \textbf{Psyquandle Coloring Quivers}}

\author{Jose Ceniceros\footnote{Email: jcenicer@hamilton.edu}\and
Anthony Christiana \footnote{Email: achristi@hamilton.edu}\and
Sam Nelson\footnote{Email: Sam.Nelson@cmc.edu. 
Partially supported by Simons Foundation collaboration grant 702597.}}

\begin{document}
\maketitle

\begin{abstract}
We enhance the psyquandle counting invariant for singular knots and pseudoknots
using quivers analogously to quandle coloring quivers. This enables us to extend
the in-degree polynomial invariants from quandle coloring quiver theory to the 
case of singular knots and pseudoknots. As a side effect we obtain biquandle 
coloring quivers and in-degree polynomial invariants for classical and virtual
knots and links.
\end{abstract}

\parbox{5.5in} {\textsc{Keywords:} Psyquandles, biquandles, quivers, 
pseudoknots, singular knots, virtual knots

                \smallskip
                
                }

\section{Introduction}

\textit{Singular knots and links} are four-valent spatial graphs in which the vertices (known as \textit{singular crossings}) have a fixed cyclic ordering. \textit{Pseudoknots} are classical knots in which crossing information for certain crossings, known as \textit{precrossings}, is unknown. The Reidemeister moves for singular knots and pseudoknots are nearly the same -- more precisely, replacing the singular crossings with precrossings transforms all of the singular Reidemeister moves into pseudoknot Reidemeister moves, with only one move missing, a Reidemeister I-style move with a precrossing.

In \cite{NOS}, an algebraic structure known as \textit{psyquandle} was introduced and applied to define invariants of singular knots and pseudoknots. A psyquandle is a set with four binary operations satisfying algebraic axioms coming from the Reidemeister moves for pseudoknots and singular knots. Psyquandles can be regarded as classical biquandles enhanced with two additional binary operations whose interactions with each other and the classical biquandle operations encode the pseudoknot and singular Reidemeister moves.

In \cite{CN}, the quandle counting invariant for classical knots and links was enhanced with a choice of subset of the endomorphism ring of the coloring quandle to define a quiver-valued invariant of knots and links known as the \textit{quandle coloring quiver}. Since comparing isomorphism classes of quivers directly can be inconvenient, a polynomial invariant called the \textit{in-degree polynomial} was defined from the quandle coloring quiver. 

In this paper we apply the idea of coloring quivers to the cases of biquandles and psyquandles, defining new directed-graph valued invariants of classical knots as well as pseudoknots and singular knots. The in-degree polynomials are also defined for these structures as an application. The paper is organized as follows. In Section \ref{K} we provide a brief review of the basics of virtual knots, singular knots and pseudoknots. In Section \ref{PBR} we review the basics of psyquandles and biquandles. In Section \ref{PCQ} we define psyquandle coloring quivers. In Section \ref{BCQ} we define a special case of psyquandle coloring quivers, biquandle coloring quivers, for classical and virtual knots and links. In Section \ref{E} we provide examples of the new invariants, in particular establishing that the quivers and their associated in-degree polynomials are not determined by the counting invariant in general. We conclude in Section \ref{Q} with some questions for future research.

\section{Virtual Knots, Singular Knots and Pseudoknots}\label{K}

In this section we briefly review the basics of virtual knots, singular knots 
and pseudoknots. For more details see \cite{GPV,KK,K}.

\textit{Virtual knots} can be understood as equivalence classes of 
\textit{signed Gauss codes} under the signed Gauss code versions of 
Reidemeister 
moves. Signed Gauss codes are text-based encodings of oriented knot diagrams 
formed by assigning each crossing a number and crossing sign, then
recording the sequence of over- and under-crossing labels with signs as one
travels around the knot. Unlike knot diagrams, Gauss codes do not include
easily accessible information about which strands of the knot are adjacent
in the plane, and hence the Gauss code version of the Reidemeister II move
allows for the move to be performed on non-adjacent strands, resulting
in knot diagrams which cannot be drawn on the plane but require a supporting
surface of nonzero genus. This is equivalent to an earlier idea known as 
\textit{abstract knot theory} in which abstract knots are Reidemeister
equivalence classes of collections of crossings with labels indicating how the
strands are connected. Abstract knots can be embedded
in orientable supporting surfaces of various genus; such supporting surfaces 
are defined only up to \textit{stabilization moves} which add or remove 
handles away from the knot.

For the purpose of drawing virtual knots on flat paper, we allow 
\textit{virtual crossings} which show where handles are flattened into the
paper. 
\[\includegraphics{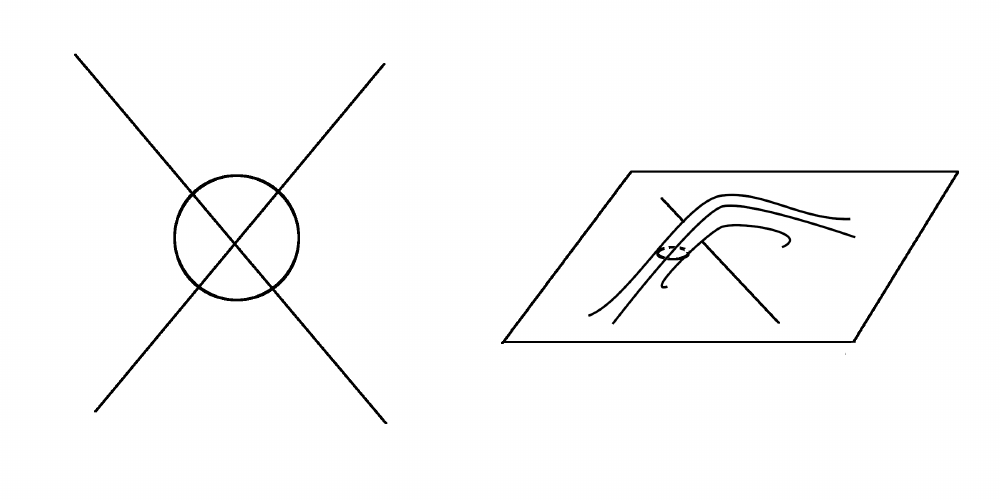}\]
These crossings interact with classical crossings via the \textit{detour move}
\[\includegraphics{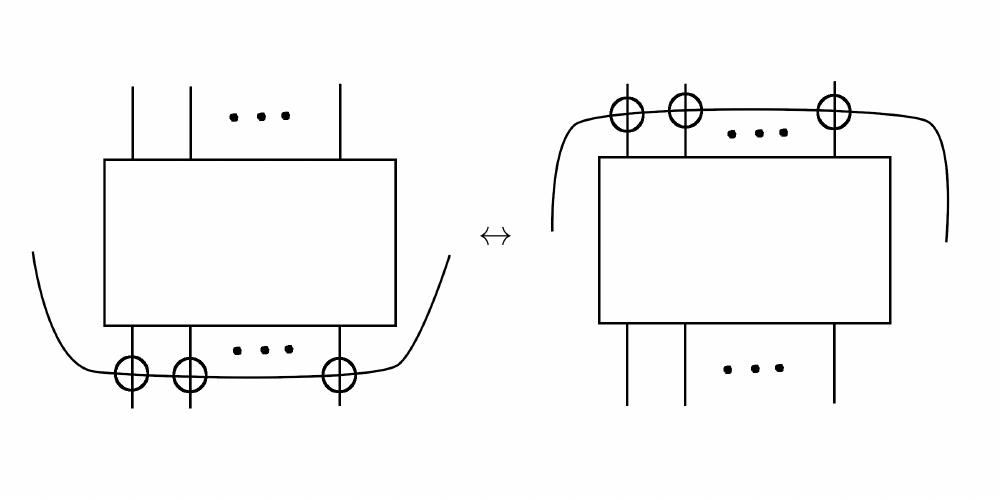}\]
motivated by the observation that repositioning of a handle before 
flattening into the plane does not change the virtual knot type. The detour 
move can be broken down into four cases known as the \textit{virtual 
Reidemeister moves}.
\[\includegraphics{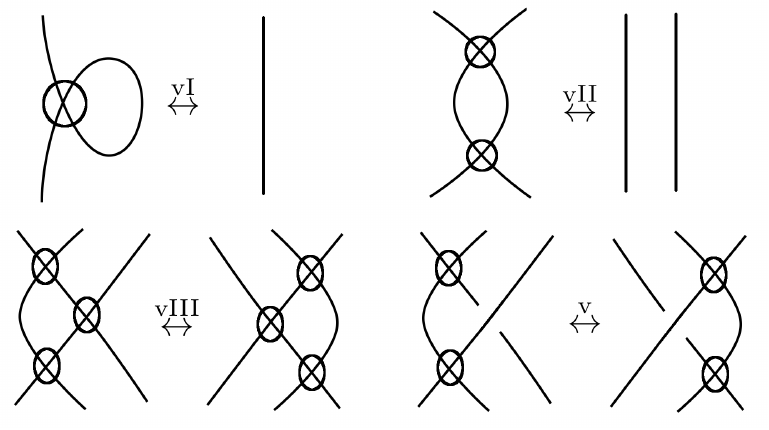}\]

A \textit{singular knot} has \textit{singular crossings} which are 4-valent 
vertices with a fixed cyclic ordering, known as \textit{rigid vertices}. 
Singular crossings interact with classical crossings via the \textit{singular 
Reidemeister moves}.
\[\scalebox{0.85}{\includegraphics{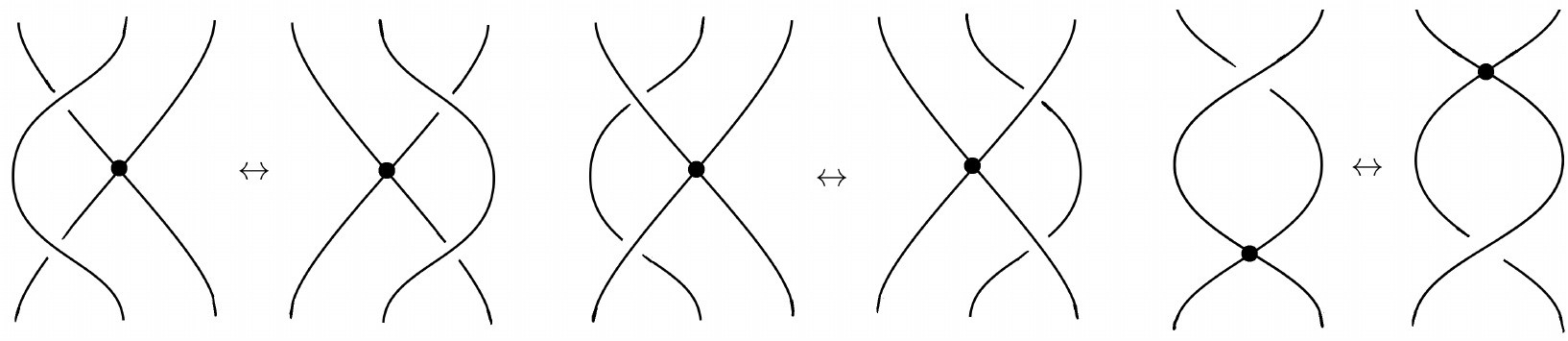}}\]

Singular knots are similar to 4-valent \textit{spatial graphs}, which may be 
understood as a quotient of singular knots with an extra move allowing 
edges at a 4-valent vertex to ``swivel'':
\[\scalebox{0.8}{\includegraphics{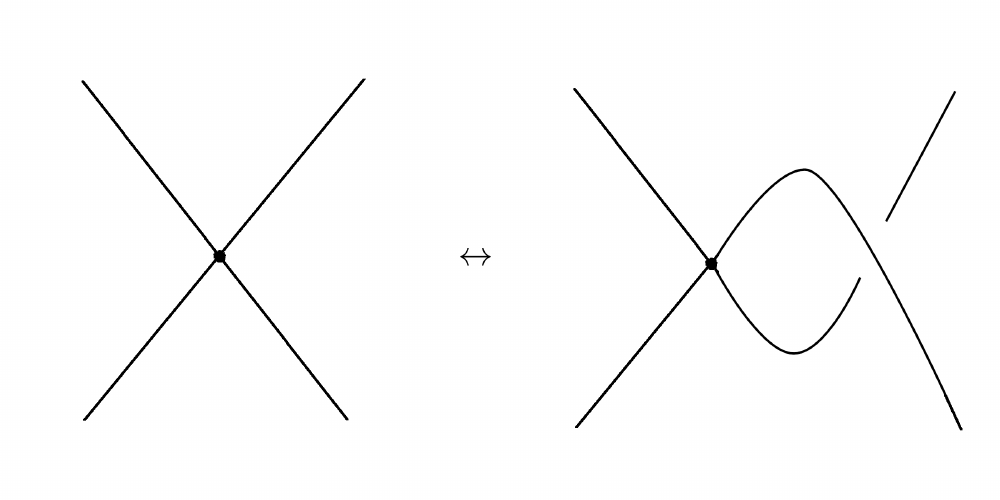}}\]
In this paper we will not allow the swivel move; all of our vertices will be 
rigid.

\textit{Pseudoknots} have \textit{precrossings} which are classical crossings
whose crossing sign is unknown. The idea originated in Biology, where pictures
of tiny physical knotted structures such as DNA strands enable some but not all
crossings to be resolved; see for example \cite{CDDK,EP,H,HHJJMR,LR,QR}. 
A precrossing may be conceptualized as a probability
distribution with the two resolutions as outcomes. Precrossings interact with
classical crossings via the \textit{pseudoknot Reidemeister moves}, which are 
the same as the singular Reidemeister moves with singular crossing replaced 
with precrossings plus one additional move, the pI move:
\[\includegraphics{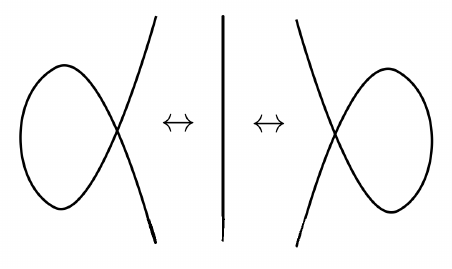}.\]

\section{Psyquandles and Biquandles}\label{PBR}

We begin with a definition from \cite{NOS}, slightly restated as in \cite{CN2}. We note that \cite{CN2} contains a misprint in axiom (i), where both $y$s should be $x$s.

\begin{definition}
Let $X$ be a set. A \textit{psyquandle structure} on $X$ consists of four binary
operations $\utr,\otr,\ud$ and $\od$ satisfying the following axioms:
\begin{itemize}
\item[(0)] The operations $\utr,\otr,\ud,\od$ are all right-invertible, i.e., there are operations $\utr^{-1}, \otr^{-1}, \ud^{-1}, \od^{-1}$ such that for all $x,y\in X$ we have
\[\begin{array}{rcccl}
(x\utr y)\utr^{-1} y & = & x & = & (x\utr^{-1}y)\utr y \\ 
(x\otr y)\otr^{-1} y & = & x & = & (x\otr^{-1}y)\otr y \\ 
(x\ud y)\ud^{-1} y & = & x & = & (x\ud^{-1}y)\ud y \ \mathrm{and}\\ 
(x\od y)\od^{-1} y & = & x & = & (x\od^{-1}y)\od y, \\ 
\end{array}\]
\item[(i)] For all $x\in X$, $x\utr x=x\otr x$,
\item[(ii)] For all $x,y\in X$, the maps $S,S':X\times X\to X\times X$ given by 
\[S(x,y)=(y\otr x,x\utr y) \quad\mathrm{and}\quad S'(x,y)=(y\od x,x\ud y)\] 
are invertible,
\item[(iii)] For all $x,y,z\in X$ we have the \textit{exchange laws}
\[\begin{array}{rcl}
(x\utr y)\utr(z\utr y) & = & (x\utr z)\utr(y\otr x) \\
(x\utr y)\otr(z\utr y) & = & (x\otr z)\utr(y\otr x) \\
(x\otr y)\otr(z\otr y) & = & (x\otr z)\otr(y\utr x), \\
\end{array}\]
\item[(iv)] For all $x,y\in X$ we have
\[\begin{array}{rcl}
x\ud((y\otr x)\od^{-1} x) & = & [(x\utr y)\od^{-1} y]\otr[(y\otr x)\ud^{-1} x]\\
y\ud((x\utr y)\od^{-1} y) & = & [(y\otr x)\od^{-1} x]\utr[(x\utr y)\od^{-1} y],
\end{array}\]
and
\item[(v)] For all $x,y,z\in X$ we have
\[\begin{array}{rcl}
(x\otr y)\otr (z\od y) & = & (x\otr z)\otr (y\ud z) \\
(x\utr y)\utr (z\od y) & = & (x\utr z)\utr (y\ud z) \\
(x\otr y)\od (z\otr y) & = & (x\od z)\otr (y\utr z) \\
(x\utr y)\ud (z\utr y) & = & (x\ud z)\utr (y\otr z) \\
(x\otr y)\ud (z\otr y) & = & (x\ud z)\otr (y\utr z) \\
(x\utr y)\od (z\utr y) & = & (x\od z)\utr (y\otr z).
\end{array}\]
\end{itemize}
A psyquandle which also satisfies 
\begin{itemize}
\item[(vi)] For all $x\in X$ we have 
\[x\ud x=x\od x\]
\end{itemize}
is said to be \textit{pI-adequate}.
\end{definition}

\begin{definition}
A set $X$ with operations $\utr, \otr$ satisfying conditions (0)-(iii) 
is a \emph{biquandle}.
\end{definition}

\begin{example}
Let $X$ be a module over the ring $\mathbb{Q}[t^{\pm 1},s^{\pm 1}]$ of 
Laurent polynomials in two variables with rational coefficients. 
Then $X$ is a pI-adequate psyquandle, known as a \textit{Jablan Psyquandle}, 
under the operations
\[\begin{array}{rcl}
x\utr y & = & tx+ (s-t) y \\
x\otr y & = & sx \\
x\ud y & = & \frac{s+t}{2} x+ \frac{s-t}{2} y\\
x\od y & = &\frac{s+t}{2} x+ \frac{s-t}{2} y.
\end{array}\]
More generally, $\mathbb{Q}$ can be replaced with any commutative ring with
identity in which $2$ is invertible, e.g. $\mathbb{Z}_9$. See \cite{NOS} for 
more.
\end{example}

\begin{example}
Given a finite set $X=\{1,2,\dots, n\}$, we can specify a psyquandle
structure on $X$ by explicitly listing the operation tables of the four
psyquandle operations. In practice it is convenient to put these together 
into an $n\times 4n$ bock matrix, so the psyquandle structure on $X=\{1,2,3\}$
specified by
\[
\begin{array}{r|rrr} \utr & 1 & 2 & 3 \\ \hline 1 & 2 & 2 & 2 \\ 2 & 3 & 3 & 3 \\ 3 & 1 & 1 & 1\end{array}\ \ 
\begin{array}{r|rrr} \otr & 1 & 2 & 3 \\ \hline 1 & 2 & 2 & 2 \\ 2 & 3 & 3 & 3 \\ 3 & 1 & 1 & 1\end{array}\ \ 
\begin{array}{r|rrr} \ud & 1 & 2 & 3 \\ \hline 1 & 3 & 3 & 3 \\ 2 & 1 & 1 & 1 \\ 3 & 2 & 2 & 2\end{array}\ \ 
\begin{array}{r|rrr} \od & 1 & 2 & 3 \\ \hline 1 & 3 & 3 & 3 \\ 2 & 1 & 1 & 1 \\ 3 & 2 & 2 & 2\end{array}
\]
is encoded as the block matrix
\[
\left[\begin{array}{rrr|rrr|rrr|rrr}
2 & 2 & 2 & 2 & 2 & 2 & 3 & 3 & 3 & 3 & 3 & 3 \\
3 & 3 & 3 & 3 & 3 & 3 & 1 & 1 & 1 & 1 & 1 & 1 \\
1 & 1 & 1 & 1 & 1 & 1 & 2 & 2 & 2 & 2 & 2 & 2
\end{array}\right].
\]
Such a psyquandle is pI-adequate if and only if the diagonals of the 3rd and 4th matrices
are equal.
\end{example}

The psyquandle axioms are motivated by the semiarc coloring rules for
singular knots and pseudoknots below:
\begin{equation}\raisebox{-0.5in}{\includegraphics{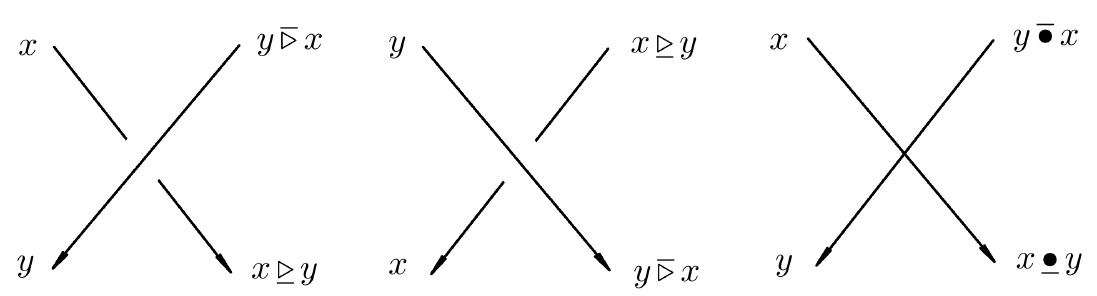}}\label{eq1}\end{equation}
The relationship between the psyquandle operations and their right inverses is 
illustrated below:
\begin{equation*}\raisebox{-0.5in}{\includegraphics{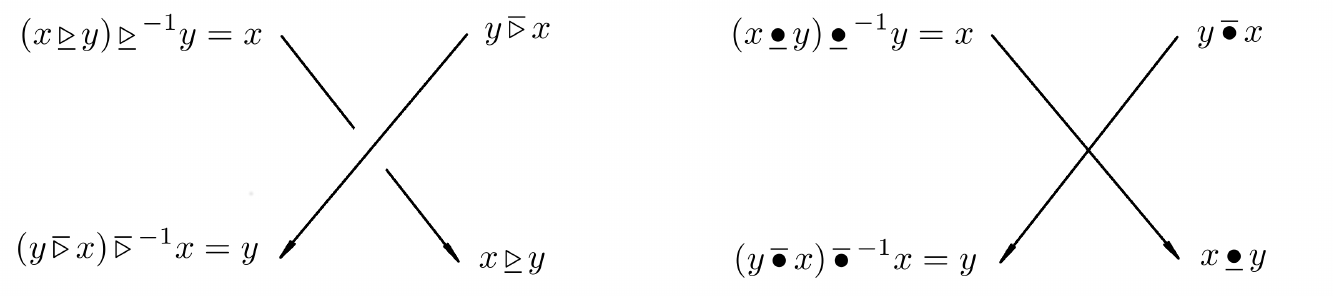}}\end{equation*}

\begin{definition}
Let $X$ be a finite psyquandle (respectively, a $pI$-adequate finite 
psyquandle). An \textit{$X$-coloring} of an oriented singular 
knot or link diagram $L$ (respectively, an oriented pseudoknot or 
pseudolink diagram $L$) is an assignment of an element of $X$ to each semiarc
in $L$ such that the coloring rules in (\ref{eq1}) are
satisfied at every crossing.
\end{definition}

\begin{definition}
Let $X,Y$ be finite psyquandles with operations denoted by $\utr_*, \otr_*, 
\ud_*$ and $\od_*$ for $*\in\{X,Y\}$. A map $f: X \to Y$ 
is  a \emph{psyquandle homomorphism} if the following conditions are satisfied 
for all $x, y \in X:$
\begin{eqnarray}
f(x\utr_X y ) & = &  f(x) \utr_Y f(y),\\
f(x\otr_X y ) & = & f(x) \otr_Y f(y),\\
f(x\ud_X y ) & = & f(x) \ud_Y f(y)\quad \mathrm{and}\\
f(x\od_X y ) & = &f(x) \od_Y f(y).
\end{eqnarray}
If $f:X \rightarrow X$ satisfies conditions (2)-(5), then $f$ is a \emph{psyquandle endomorphism}.
\end{definition}

Psyquandles are useful for defining invariants of oriented singular knots and links. For every oriented singular or pseudolink diagram $L$ there is an associated \emph{fundamental psyquandle of $L$} denoted by $\mathcal{P}(L)$ or $\mathcal{P}_I(L)$ for the \emph{fundamental pI-adequate psyquandle of $L$}. The fundamental psyquandle of $L$ is defined as the quotient of the free psyquandle on the set of semiarcs in the diagram $L$ by the equivalence relation generated by the relations determined by the crossings in the diagram $L$. For details of the construction of the fundamental psyquandle, see \cite{NOS}.

\begin{theorem}\cite{NOS}\label{counting}
Let $X$ be a finite psyquandle and let $L$ be an oriented singular knot or link diagram. Then the cardinality of the set of
$X$-colorings of $L$, 
\[\Phi_X^{\mathbb{Z}}(L)=|\mathrm{Hom}(\mathcal{P}(L),X)| \]
is an integer-valued invariant of singular knots and links.
\end{theorem}

Note that each element of $\mathrm{Hom}(\mathcal{P}(L),X)$ can be thought of as a coloring of $L$ which associates an element of $X$ to each semiarc in $L$ such that the psyquandle operations in $X$ are compatible with the relations determined by the crossings in $L$. The invariant defined in Theorem~\ref{counting} counts the number of $X$-colorings of $L$, therefore, the invariant is commonly known as the \emph{psyquandle counting invariant}. If $L$ is a diagram of an oriented link there is an associated \emph{fundamental biquandle of $L$}, denoted $\mathcal{BQ}(L)$. Furthermore, the fundamental biquandle of $L$ can be used to define the following \emph{biquandle counting invariant}, $\Phi_X^\mathbb{Z}(L) = | \mathrm{Hom}(\mathcal{BQ}(L),X)|$, for any finite biquandle $X$.

\section{Psyquandle Coloring Quivers} \label{PCQ}

In this section, we define quivers (i.e., directed graphs) associated to a triple $(L,X,S)$ consisting of an oriented singular link (respectively, pseudolink) $L$, a finite psyquandle (respectively, a pI-adequate psyquandle) $X$, and subset $S \subset \text{Hom}(X,X)$ of the ring of endomorphisms of $X$. 
This is a generalization of the quandle coloring quivers found in \cite{CN} associated to the triple $(L,X,S)$ where $L$ is an oriented link, $X$ is a finite quandle and $S\subset\mathrm{Hom}(X,X)$ is a set of quandle endomorphisms.

\begin{definition}
Let $X$ be a finite psyquandle and $L$ an oriented singular link. For any set of psyquandle endomorphisms $S \subset \mathrm{Hom}(X,X)$, the associated \emph{psyquandle coloring quiver}, denoted by $\mathcal{PQ}_X^S(L)$, is the directed graph with a vertex for every element $f \in \mathrm{Hom}(\mathcal{P}(L), X)$ and an edge directed from $f$ to $g$ where $g = \phi f$ for an element $\phi \in S$. In the case when $S = \mathrm{Hom}(\mathcal{P}(L),X)$ we will call this the \emph{full psyquandle coloring quiver} of $L$ with respect to $X$, denoted by $\mathcal{PQ}_X(L)$. Furthermore, in the case when $S=\{ \phi\}$ is a singleton, we will denote the quiver associated to the singular link and psyquandle by $\mathcal{PQ}_X^\phi(L)$.
\end{definition}

In the case where $X$ is a finite pI-adequate psyquandle with $S\subset \mathrm{Hom}(X,X)$ and $L$ is an oriented pseudolink, we will still denote the \emph{pI-adequate psyquandle coloring quiver} of $L$ with respect to $X$ by $\mathcal{PQ}_X^S(L)$; we note that for a psyquandle to define invariants for pseudolinks, it must be pI-adequate. 

\begin{theorem}
Let $X$ be a finite psyquandle, $S \subset \mathrm{Hom}(X,X)$ and $L$ an oriented singular knot. Then the quiver $\mathcal{PQ}_X^S(L)$ is an invariant of $L$.
\end{theorem}

\begin{proof}
We observe that for any psyquandle coloring $f \in \mathrm{Hom}(\mathcal{P}(L), X)$ and $\phi:X\to X$, the conditions required for $g = \phi f$ to be a valid psyquandle coloring are precisely the conditions for $f$ to be a psyquandle endomorphism. Since the quiver $\mathcal{PQ}_X^S$ is determined up to isomorphism by $X$ and $\textup{Hom}(\mathcal{P}(L),X)$, the quiver is an invariant of $L$.
\end{proof}

\begin{corollary}
Any invariant of directed graphs applied to $\mathcal{PQ}_X^S(L)$ defines an invariant of oriented singular links. 
\end{corollary}

\begin{example}\label{Qui1}
Consider the psyquandle $X$ with operation matrix
\[ \left[
\begin{array}{cccc|cccc|cccc|cccc}
 1 & 1 & 1 & 1 &  1 & 1 & 1 & 1 &  1 & 1 & 1 & 1 &  1 & 1 & 1 & 1\\
 3 & 3 & 3 & 3 &  3 & 3 & 3 & 3 &  2 & 2 & 2 & 2 &  2 & 2 & 2 & 2\\
 2 & 2 & 2 & 2 &  2 & 2 & 2 & 2 &  3 & 3 & 3 & 3 &  3 & 3 & 3 & 3\\
 4 & 4 & 4 & 4 &  4 & 4 & 4 & 4 &  4 & 4 & 4 & 4 &  4 & 4 & 4 & 4\\
\end{array}
\right]. 
\]
The \emph{2-bouquet graph of type L} listed as $1^l_1$ in \cite{O}
\[
\includegraphics[scale=1]{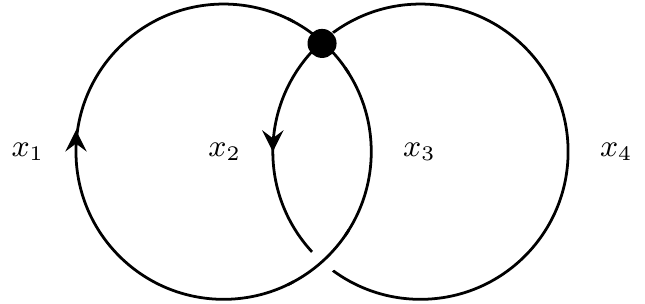}
\]
%
%
%
%
has 4 colorings by the psyquandle $X$, each of which we can identify as a 4-tuple $(f(x_1),f(x_2),f(x_3),f(x_4)):$
\[ \textup{Hom}(\mathcal{P}(1^l_1),X) =\lbrace (1, 1, 1, 1), (1, 4, 1, 4), (4, 1, 4, 1), (4, 4, 4, 4)\rbrace.\]
The psyquandle endomorphism $\phi: X \rightarrow X$ defined by $\phi(1)=\phi(4)=1$ and $\phi(2)=\phi(3)=4$ yields the following psyquandle quiver $\mathcal{PQ}_X^{\phi}(1^l_1)$

\[
\begin{tikzpicture}[node distance=2cm, auto]
  \node (P) {$(1, 4, 1, 4)$};
  \node (B) [right of=P] {$(4, 1, 4, 1)$};
  \node (A) [right of=B] {$(4, 4, 4, 4)$};
  \node (C) [below of=B] {$(1, 1, 1, 1)$};
  \draw[->] (P) to node  {} (C);
  \draw[->] (B) to node  {} (C);
  \draw[->] (A) to node  {} (C);
  \draw[->] (C) to [out=330,in=210,looseness=5] (C);
\end{tikzpicture}
\hspace{2cm}
\begin{tikzpicture}
\centering
\node at (0,1.5) {or};
\node at (0,0) {};
\end{tikzpicture}
\hspace{2cm}
\includegraphics[scale=.5]{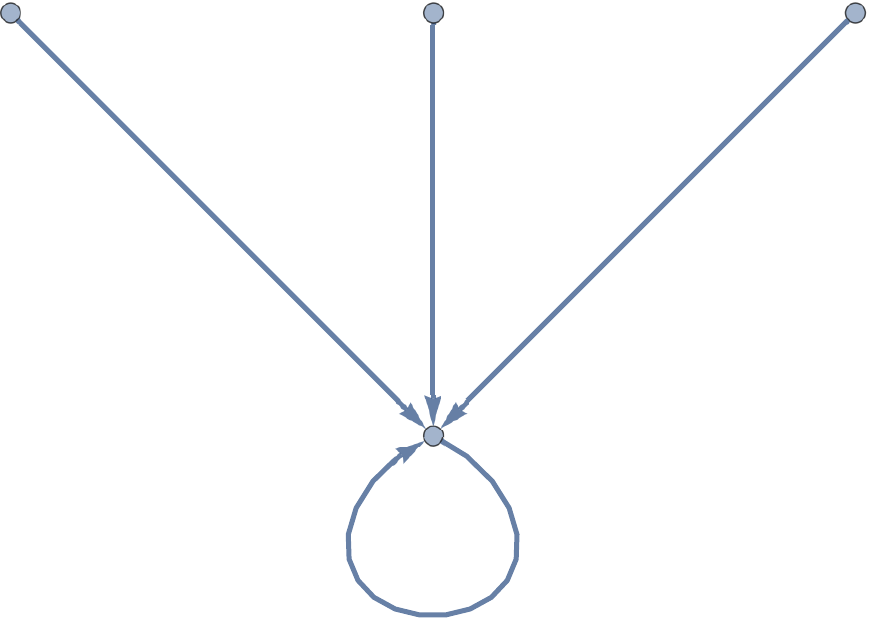}
\]
\end{example}

\begin{example}\label{Qui2}
Note that the psyquandle in Example~\ref{Qui1} is pI-adequate since the two right blocks have the same diagonal. The pseudoknot $3_1.1$ listed in \cite{HHJJMR}

\[\includegraphics[scale=1]{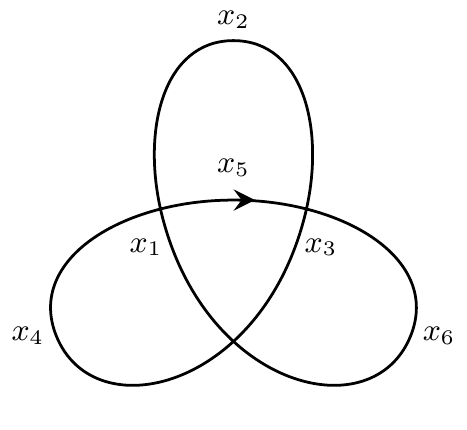} \]
%
%
has 4 colorings by the pI-adequate psyquandle $X$, each of which we can identify as a 6-tuple $(f(x_1),f(x_2),\allowbreak f(x_3),f(x_4),f(x_5),f(x_6)):$
\[ \textup{Hom}(\mathcal{P}_I(3_1.1),X) =\lbrace (1, 1, 1, 1, 1, 1),(2, 2, 2, 2, 2, 2),(3, 3, 3, 3, 3, 3),(4, 4, 4, 4, 4, 4)\rbrace.\]
Setting $S=\lbrace \phi_1, \phi_2 \rbrace$ where $\phi_1, \phi_2: X \rightarrow X$ are psyquandle endomorphisms defined by $\phi_1(1)=\phi_1(2)=\phi_1(3)=1$ and $\phi_1(4)=4$ and $\phi_2(1)=\phi_2(2)=\phi_2(3)=4$ and $\phi_2(4)=1$ yields the following pI-adequate psyquandle quiver $\mathcal{PQ}_{X}^{S}(3_1.1)$

\[
\begin{tikzpicture}[node distance=2cm, auto]
  \node (P) {$(2, 2, 2, 2, 2, 2)$};
  \node (B) [right of=P, above of=P] {$(1, 1, 1, 1, 1, 1)$};
  \node (A) [right of=B, below of=B] {$(3, 3, 3, 3, 3, 3)$};
  \node (C) [right of=P, below of=P] {$(4, 4, 4, 4, 4, 4)$};
  \draw[->] (P) to node  {} (B);
  \draw[->] (P) to node  {} (C);
  \draw[->] (A) to node  {} (B);
  \draw[->] (A) to node  {} (C);
  \draw[->] (B) to [out=150,in=30,looseness=5] (B);
  \draw[->] (C) to [out=330,in=210,looseness=5] (C);
  \draw[->] (C) to [out=100,in=260,looseness=1] (B);
  \draw[->] (B) to [out=280,in=80,looseness=1] (C);
\end{tikzpicture}
\hspace{1cm}
\begin{tikzpicture}
\centering
\node at (0,1.5) {or};
\node at (0,0) {};
\end{tikzpicture}
\hspace{1cm}
\includegraphics[scale=.7]{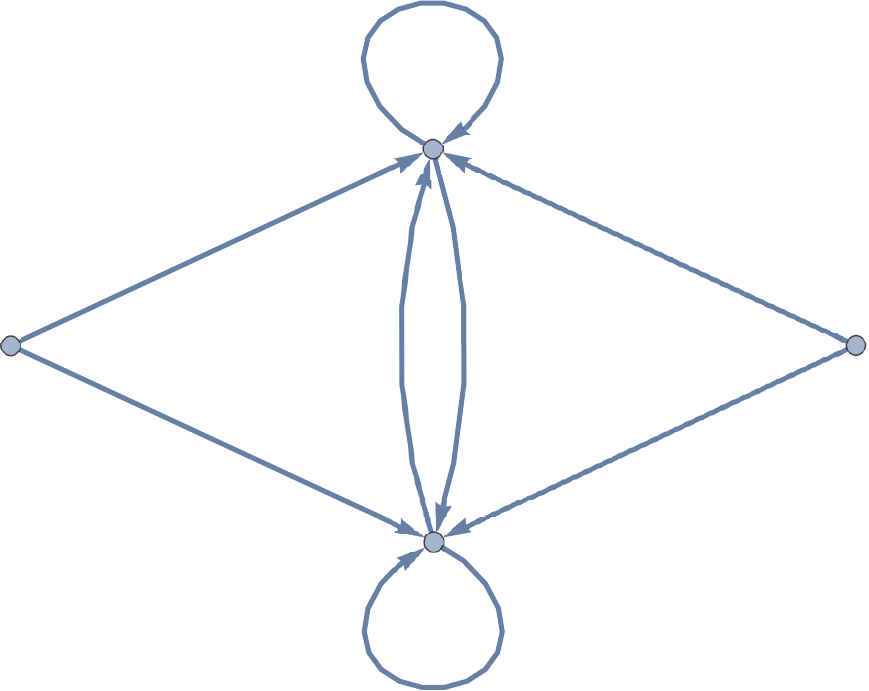}
\]
\end{example}

\section{Biquandle Coloring Quivers}\label{BCQ}

We have defined psyquandle coloring quiver invariants for singular knots and 
links and for pseudoknots in the case of pI-adequate psyquandles. However, 
these invariants are actually more versatile than they may initially appear:

\begin{itemize}
\item A classical knot or link can be regarded as a singular knot or link with
an empty set of singular crossings. In this case, for any psyquandle $X$ with $S \in \textup{Hom}(X,X)$, 
colorings simply do not use the $\ud,\od$ operations and hence we have the 
\textit{biquandle coloring quiver invariant}, denoted by $\mathcal{BQ}_X^S$.
\item A virtual knot or link can similarly be regarded as a virtual singular
knot with empty set of singular crossings. By the usual technique of ignoring
virtual crossings for coloring purposes, we obtain biquandle coloring quivers
for virtual knot and links.
\item Extending virtual knot theory to the cases of virtual singular 
knots and virtual pseudoknots respectively by allowing singular crossings 
and precrossings respectively in the virtual detour move, psyquandle
coloring quivers and in-degree polynomials define invariants for these objects
as well.
\end{itemize}

\begin{example}
Consider the biquandle $X$ with operation matrix given the left two blocks of the operation matrix in Example~\ref{Qui1}. The classical right-handed trefoil $3_1$

\[ \includegraphics[scale=1]{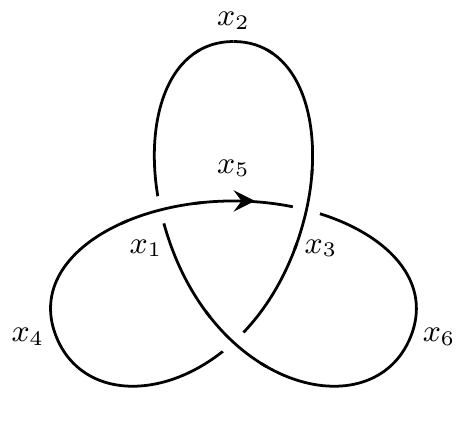} \]
%
has 4 colorings by biquandle $X$, each of which we can identify as a 6-tuple $(f(x_1),f(x_2),\allowbreak f(x_3),f(x_4),f(x_5),f(x_6)):$
\[ \textup{Hom}(\mathcal{BQ}(3_1),X) =\lbrace (1, 1, 1, 1, 1, 1),(2, 3, 2, 3, 2, 3),(3, 2, 3, 2, 3, 2),(4, 4, 4, 4, 4, 4)\rbrace.\]
The biquandle endomorphism $\phi: X \rightarrow X$ defined by $\phi(1)=\phi(4)=1$ and $\phi(2)=\phi(3)=4$ yields the following biquandle quiver $\mathcal{BQ}_X^{\phi}(3_1)$
\[
\begin{tikzpicture}[node distance=2.2cm, auto]
  \node (P) {$(2,3,2,3,2,3)$};
  \node (B) [right of=P] {$(3,2,3,2,3,2)$};
  \node (A) [below of=B] {$(4,4,4,4,4,4)$};
  \node (C) [below of=A] {$(1,1,1,1,1,1)$};
  \draw[->] (P) to node  {} (A);
  \draw[->] (B) to node  {} (A);
  \draw[->] (A) to node  {} (C);
  \draw[->] (A) to node  {} (C);
  \draw[->] (C) to [out=330,in=210,looseness=5] (C);
\end{tikzpicture}
\hspace{1cm}
\begin{tikzpicture}
\centering
\node at (0,2.5) {or};
\node at (0,0) {};
\end{tikzpicture}
\hspace{1cm}
  \includegraphics[scale=.5]{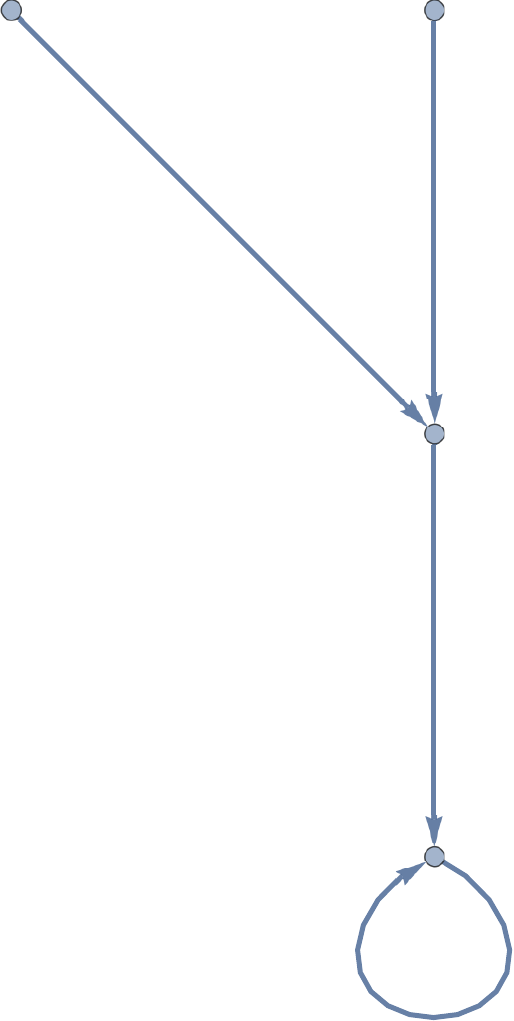}
\]
\end{example}

Note that in a psyquandle coloring quiver (respectively, biquandle coloring quiver), the number of edges out of a vertex will correspond to the cardinality of $S$. On the other hand, the number of edges in to a vertex may be different. We will refer to the number of edges in to a vertex as the \emph{in-degree} of the vertex and will denoted by $\text{deg}^+$. We can use the in-degree of each vertex of a psyquandle coloring quiver (respectively, biquandle coloring quiver) to define the following polynomial invariant of links.


\begin{definition}
Let $X$ be a finite psyquandle (respectively, a pI-adequate psyquandle), $S \subset \text{Hom}(X,X)$ a set of psyquandle endomorphisms, $L$ an oriented singular link (respectively, pseudolink) and $\mathcal{PQ}_X^{S}(L)$ the associated psyquandle coloring quiver of $L$ (respectively, pI-adequate psyquandle quiver of $L$) with set of vertices $V(\mathcal{PQ}_X^S(L))$. Then the \emph{in-degree quiver polynomial} of $L$ with respect to $X$ is 
\[ \Phi_X^{\text{deg}^+, S}(L) = \sum_{f \in V(\mathcal{PQ}_X^S(L))}u^{\text{deg}^+(f)}. \]
In the case that $X$ is a finite biquandle and $L$ is an oriented link or virtual link, then the \emph{in-degree quiver polynomial} of $L$ with respect to $X$ is 
\[\Phi_{X}^{\text{deg}^+, S}(L) = \sum_{f \in V(\mathcal{BQ}_{X}^S(L))}u^{\text{deg}^+(f)}.\]
If $S= \{ \phi \}$ is a singleton we will denote the in-degree quiver polynomial of $L$ as $\Phi_X^{deg^+,\phi}(L)$ and if $S= \mathrm{Hom}(X,X)$ we will denote the in-degree polynomial by $\Phi_X^{deg^+}(L)$.
\end{definition} 

By construction, we have

\begin{corollary}
The in-degree quiver polynomials are invariants of singular knots and links in the case of psyquandles, of pseudoknots in the case of PI-adequate psyquandles, and of classical and virtual knots and links in the case of biquandles.
\end{corollary}

\section{Examples}\label{E}

In this section we collect some additional examples and computations of the
new invariants. The examples were computed using our custom \texttt{python} code.

\begin{example}\label{inoutpolyex1}
Let $X$ be the psyquandle with the following operation matrix 

\[\left[
\begin{array}{cccccccc|cccccccc|cccccccc|cccccccc}
 1 & 4 & 2 & 3 & 3 & 2 & 1 & 4  &  1 & 1 & 1 & 1 & 1 & 1 & 1 & 1  &  1 & 4 & 2 & 3 & 1 & 1 & 1 & 1  &  1 & 1 & 1 & 1 & 3 & 2 & 1 & 4 \\
 3 & 2 & 4 & 1 & 4 & 1 & 2 & 3  &  2 & 2 & 2 & 2 & 2 & 2 & 2 & 2  &  3 & 2 & 4 & 1 & 2 & 2 & 2 & 2  &  2 & 2 & 2 & 2 & 4 & 1 & 2 & 3 \\
 4 & 1 & 3 & 2 & 1 & 4 & 3 & 2  &  3 & 3 & 3 & 3 & 3 & 3 & 3 & 3  &  4 & 1 & 3 & 2 & 3 & 3 & 3 & 3  &  3 & 3 & 3 & 3 & 1 & 4 & 3 & 2 \\
 2 & 3 & 1 & 4 & 2 & 3 & 4 & 1  &  4 & 4 & 4 & 4 & 4 & 4 & 4 & 4  &  2 & 3 & 1 & 4 & 4 & 4 & 4 & 4  &  4 & 4 & 4 & 4 & 2 & 3 & 4 & 1 \\
 8 & 8 & 8 & 8 & 5 & 5 & 5 & 5  &  5 & 5 & 5 & 5 & 5 & 5 & 5 & 5  &  5 & 5 & 5 & 5 & 5 & 5 & 5 & 5  &  8 & 8 & 8 & 8 & 5 & 5 & 5 & 5 \\
 5 & 5 & 5 & 5 & 6 & 6 & 6 & 6  &  6 & 6 & 6 & 6 & 6 & 6 & 6 & 6  &  6 & 6 & 6 & 6 & 6 & 6 & 6 & 6  &  5 & 5 & 5 & 5 & 6 & 6 & 6 & 6 \\
 7 & 7 & 7 & 7 & 7 & 7 & 7 & 7  &  7 & 7 & 7 & 7 & 7 & 7 & 7 & 7  &  7 & 7 & 7 & 7 & 7 & 7 & 7 & 7  &  7 & 7 & 7 & 7 & 7 & 7 & 7 & 7 \\
 6 & 6 & 6 & 6 & 8 & 8 & 8 & 8  &  8 & 8 & 8 & 8 & 8 & 8 & 8 & 8  &  8 & 8 & 8 & 8 & 8 & 8 & 8 & 8  &  6 & 6 & 6 & 6 & 8 & 8 & 8 & 8 \\
\end{array}
\right].\]
The reader can verify that the map $\phi:X \rightarrow X$ defined by setting $\phi(1)=\phi(2)= \phi(3)=\phi(4)=\phi(7)=3$ and $\phi(5)=\phi(6)=\phi(8)=7$ is a psyquandle endomorphism. 

\[ \includegraphics[scale=1]{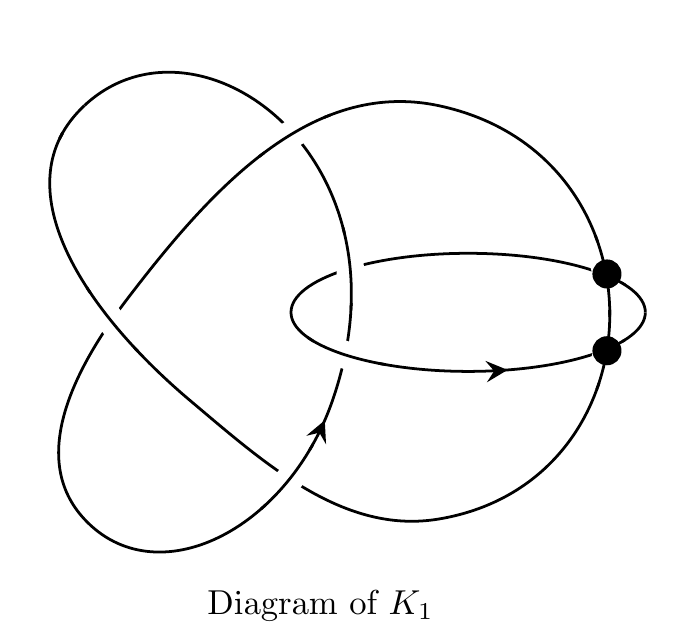} \hspace{3cm} \includegraphics[scale=1]{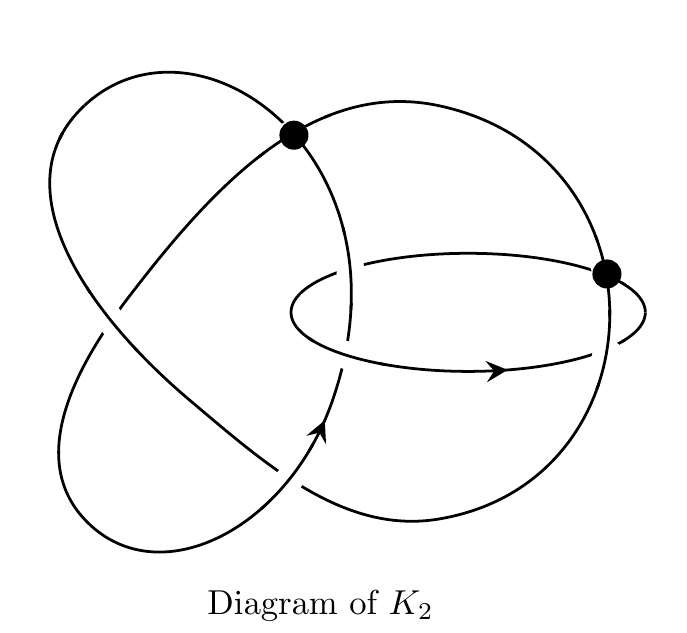} \]
Then the singular links have equal $\Phi_X^{\mathbb{Z}}(K_1) = 52 = \Phi_X^{\mathbb{Z}}(K_2)$, but the singular links are distinguished by their in-degree psyquandle quiver polynomial,
\begin{eqnarray*}
\Phi_X^{\text{deg}^+,\phi}(K_1) &=& u^{25}+u^{15}+u^{9}+u^3+48 \neq\\
 \Phi_X^{\text{deg}^+,\phi}(K_2)&=& 2u^{15}+u^{13}+u^9+48.
\end{eqnarray*}
\end{example}

\begin{example}
Consider the psyquandle in Example~\ref{inoutpolyex1}. We compute the in-degree psyquandle quiver polynomial for a choice of orientations for all prime 2-bouquet graphs of type L with up to six crossings from \cite{O}. We note that in this example we will consider the full psyquandle coloring quiver for each singular link.
\begin{center}
\begin{tabular}{ r| l }
$L$ & $\Phi_X^{\text{deg}^+}(L)$ \\
\hline
$1_1^l$	            & $u^{708} + 3u^{388} + 12u^{164} + 4u^{144} + 8u^{88} + 24u^{12}$ \\
$3_1^l$  	        & $u^{516} + 3u^{292} + 4u^{96} + 12u^{68} + 8u^{40}$	\\
$4_1^l$ 	        & $u^{708} + 3u^{388} + 12u^{164} + 4u^{144} + 8u^{88} + 24u^{12}$ 	\\	
$5_1^l$ 	        & $u^{516} + 3u^{292} + 4u^{96} + 12u^{68} + 8u^{40}$ 	\\
$5_2^l$  		    & $u^{1044} + 3u^{580} + 4u^{192} + 12u^{164} + 8u^{88} + 36u^{12}$ 	\\
$5_3^l$             & $u^{708} + 3u^{388} + 12u^{164} + 4u^{144} + 8u^{88} + 24u^{12}$  \\
$6_1^l$  		    & $u^{852} + 3u^{484} + 4u^{144} + 12u^{68} + 8u^{40} + 12u^{12}$ 	\\
$6_2^l$  		    & $u^{708} + 3u^{388} + 12u^{164} + 4u^{144} + 8u^{88} + 24u^{12}$ 	\\
$6_3^l$  		    & $u^{516} + 3u^{292} + 4u^{96} + 12u^{68} + 8u^{40}$ 	\\
$6_4^l$  		    & $u^{1044} + 3u^{580} + 4u^{192} + 12u^{164} + 8u^{88} + 36u^{12}$ 	\\
$6_5^l$  		    & $u^{1044} + 3u^{580} + 4u^{192} + 12u^{164} + 8u^{88} + 36u^{12}$ 	\\
$6_6^l$  		    & $u^{708} + 3u^{388} + 12u^{164} + 4u^{144} + 8u^{88} + 24u^{12}$ 	\\
$6_7^l$  		    & $u^{1044} + 3u^{580} + 4u^{192} + 12u^{164} + 8u^{88} + 36u^{12}$ 	\\
$6_8^l$  		    & $u^{612} + 3u^{340} + 4u^{120} + 12u^{116} + 8u^{64} + 12u^{12}$ 	\\
$6_9^l$  		    & $u^{852} + 3u^{484} + 4u^{144} + 12u^{68} + 8u^{40} + 12u^{12}$ 	\\
$6_{10}^l$  		& $u^{612} + 3u^{340} + 4u^{120} + 12u^{116} + 8u^{64} + 12u^{12}$ 	\\
$6_{11}^l$  		& $u^{612} + 3u^{340} + 4u^{120} + 12u^{116} + 8u^{64} + 12u^{12}$ 	\\
$6_{12}^l$  		& $u^{516} + 3u^{292} + 4u^{96} + 12u^{68} + 8u^{40}$.
\end{tabular}
\end{center}
\end{example}

The following examples illustrates that the in-degree psyquandle quiver polynomial is also an effective invariant of pseudolinks. Note that in this case we have to start with a pI-adequate psyquandle in order to obtain an invariant of pseudoknots and links.
\begin{example}
Let $X$ be the psyquandle with operation matrix below,
\[
\left[
\begin{array}{cccccccc|cccccccc|cccccccc|cccccccc}
 1& 4& 2& 3& 3& 2& 1& 4   &   1& 1& 1& 1& 1& 1& 1& 1   &  1& 1& 1& 1& 4& 3& 1& 2  &  1& 4& 2& 3& 4& 3& 1& 2\\
 3& 2& 4& 1& 4& 1& 2& 3   &   2& 2& 2& 2& 2& 2& 2& 2   &  2& 2& 2& 2& 3& 4& 2& 1  &  3& 2& 4& 1& 3& 4& 2& 1\\
 4& 1& 3& 2& 1& 4& 3& 2   &   3& 3& 3& 3& 3& 3& 3& 3   &  3& 3& 3& 3& 2& 1& 3& 4  &  4& 1& 3& 2& 2& 1& 3& 4\\
 2& 3& 1& 4& 2& 3& 4& 1   &   4& 4& 4& 4& 4& 4& 4& 4   &  4& 4& 4& 4& 1& 2& 4& 3  &  2& 3& 1& 4& 1& 2& 4& 3\\
 8& 8& 8& 8& 5& 5& 5& 5   &   5& 5& 5& 5& 5& 5& 5& 5   &  6& 6& 6& 6& 5& 7& 8& 6  &  6& 6& 6& 6& 5& 7& 8& 6\\
 5& 5& 5& 5& 6& 6& 6& 6   &   6& 6& 6& 6& 6& 6& 6& 6   &  8& 8& 8& 8& 8& 6& 5& 7  &  8& 8& 8& 8& 8& 6& 5& 7\\
 7& 7& 7& 7& 7& 7& 7& 7   &   7& 7& 7& 7& 7& 7& 7& 7   &  7& 7& 7& 7& 6& 8& 7& 5  &  7& 7& 7& 7& 6& 8& 7& 5\\
 6& 6& 6& 6& 8& 8& 8& 8   &   8& 8& 8& 8& 8& 8& 8& 8   &  5& 5& 5& 5& 7& 5& 6& 8  &  5& 5& 5& 5& 7& 5& 6& 8
\end{array}
\right].
\]
Note that this psyquandle is pI-adequate since the last two right blocks have the same diagonal. We will consider the following two pseudolinks diagrams which we will denote by $P_a$ and $P_b$,
\[ \includegraphics[scale=1]{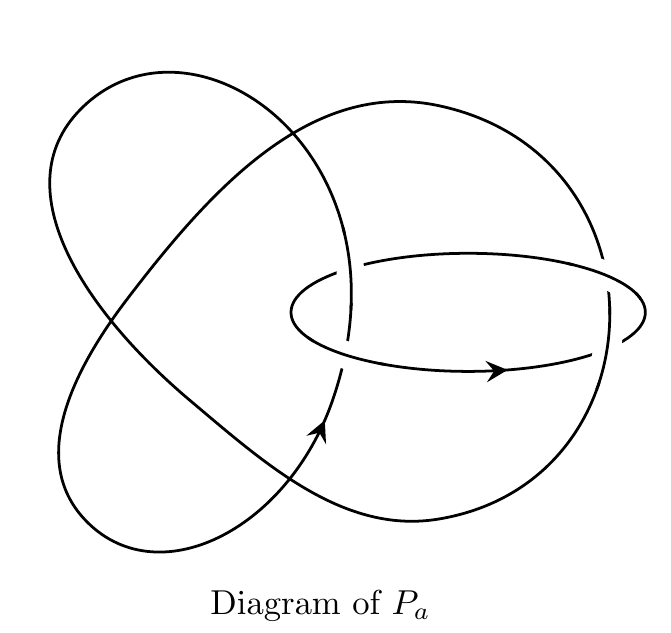} \]
\[ \includegraphics[scale=1]{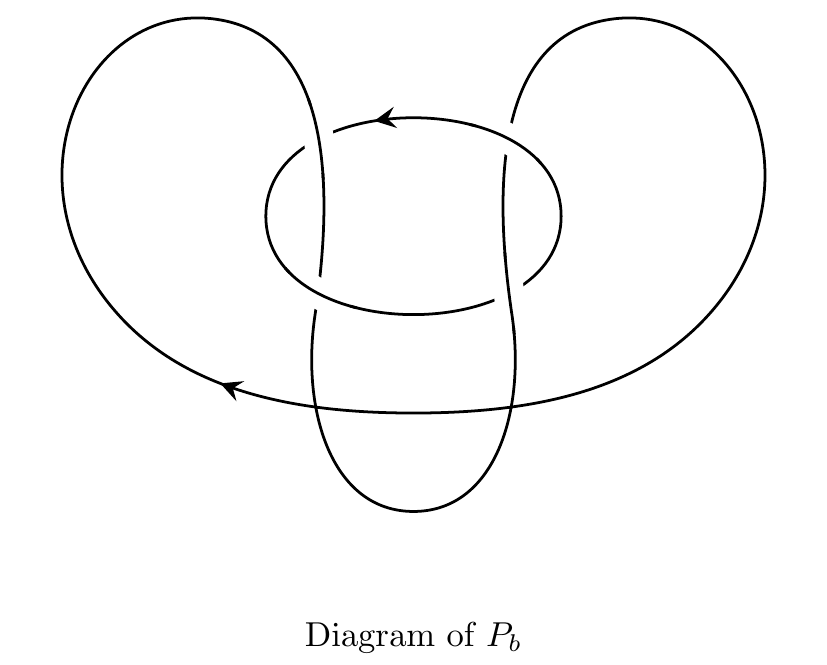} \]
The two pseudolinks diagrams $P_a$ and $P_b$ have equal psyquandle counting invariant values $\Phi_X^\mathbb{Z}(P_a) =\Phi_X^\mathbb{Z}(P_b) = 40$, but the pseudolinks are distinguished by their in-degree psyquandle quiver polynomial,
\begin{eqnarray*}
\Phi_{X}^{\text{deg}^+}(P_a) &=& u^{132}+4u^{72}+3u^{52}+8u^{32}+24u^{12} \neq\\
\Phi_{X}^{\text{deg}^+}(P_b) &=& u^{180}+4u^{84}+3u^{52}+8u^{20}+24u^{12}.
\end{eqnarray*}
Note that we are considering the full pI-adequate psyquandle coloring quiver for both $P_a$ and $P_b$.
\end{example}

Lastly, we will end this section with a few examples that illustrate that the in-degree polynomial is also useful for distinguishing classical links as well as virtual links. 

\begin{example}
Let $X$ be the biquandle with the following operation matrix

\[
\left[
\begin{array}{cccc|cccc}
3& 3& 3& 3&   3& 3& 4& 1\\
2& 2& 2& 2&   2& 2& 2& 2\\
1& 1& 1& 1&   4& 1& 1& 3\\
4& 4& 4& 4&   1& 4& 3& 4
\end{array}
\right]
\]
Setting $S = \{ \phi_1, \phi_2 \}$ where $\phi_1, \phi_2:X \rightarrow X$ are biquandle endomorphisms defined by  $\phi_1(1)=\phi_1(3)= \phi_1(4)=2$, $\phi_1(2)=4$ and $\phi_2(1)=\phi_2(3)= \phi_2(4)=4$, $\phi_2(2)=2$. The links $L7a1$ and  $L7a2$ both have equal counting invariant $\Phi_X(L7a1) = 16 = \Phi_X(L7a2)$, but the two links are distinguished by their in-degree psyquandle quiver polynomial,
\begin{eqnarray*}
\Phi_{X}^{\text{deg}^+,S}(L7a1) &=& 2u^{10}+2u^6+12 \neq\\
\Phi_{X}^{\text{deg}^+,S}(L7a2) &=& 2u^{12}+2u^4+12.
 \end{eqnarray*}
\end{example}

\begin{example}
Let $X = \mathbb{Z}_9$ be the Alexander biquandle with operations defined by $x \utr y = 7x+4y$ and $x \otr y = 2x$.
We consider the virtual knots $3.5$ and $3.7$.
\[ \includegraphics[scale=1]{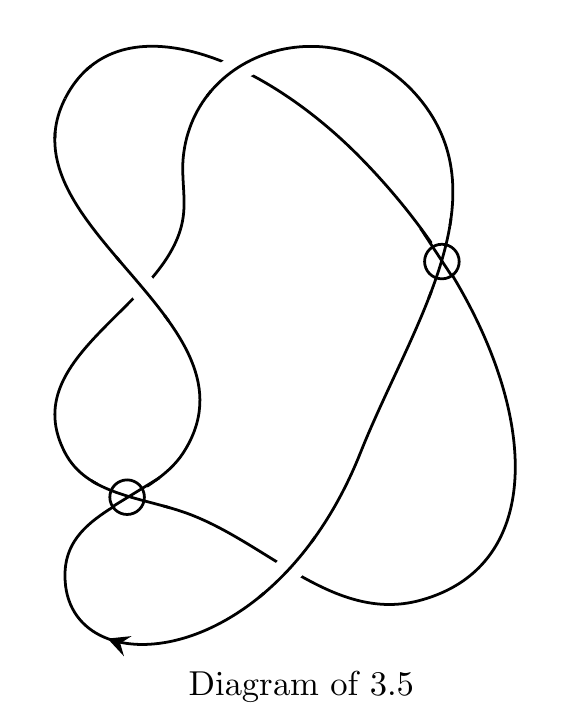} \hspace{3cm} \includegraphics[scale=1]{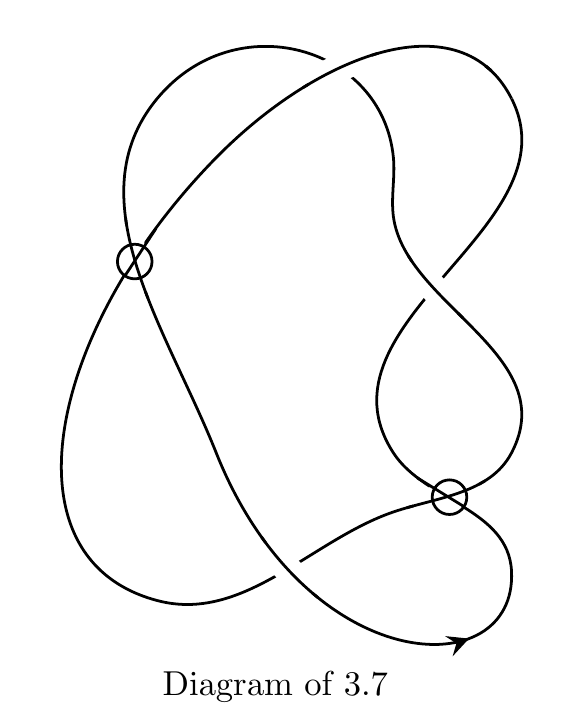}\]

Note that $3.5$ and $3.7$ have the same biquandle counting invariant values: $\Phi^{\mathbb{Z}}_{X}(3.5)=9=\Phi^{\mathbb{Z}}_{X}(3.7)$. However, the in-degree biquandle quiver polynomials can distinguish these two virtual knots,
\begin{eqnarray*}
\Phi_{X}^{\text{deg}^+}(3.5) &=& u^{21}+2u^{12}+6u^{6} \neq\\
\Phi_{X}^{\text{deg}^+}(3.7) &=& u^{33}+8u^{6}.
\end{eqnarray*}

\end{example}

\begin{example}
Let $X=\mathbb{Z}_9$ be the Alexander biquandle with operations defined by $x \utr y = 4x+y$ and $x \otr y = 5x.$ We compute the in-degree polynomial for a collection of virtual knots.
\begin{center}
\begin{tabular}{ r| l }
$L$ & $\Phi_X^{\text{deg}^+}(L)$ \\
\hline
$2.1$ & $ u^{15}+2u^6$ \\
$3.1$ & $u^{21}+2u^{12}+6u^6$ \\
$3.2$ & $u^{15}+ 2u^6$ \\
$3.3$ & $u^{15}+ 2u^6$ \\
$3.4$ & $u^{15}+ 2u^6$ \\
$3.5$ & $u^{21}+2u^{12}+6u^6$ \\
$3.6$ & $u^{51}+2u^{24}+24u^6$ \\
$3.7$ & $u^{33} + 8u^6$ \\
$4.1$ & $u^{15}+2u^6$ \\
$4.2$ & $u^{21}+2u^{12}+6u^6$ \\
$4.3$ & $u^{21}+2u^{12}+6u^6$ \\
$4.4$ & $u^{15}+2u^6$ \\
\end{tabular}
\end{center}
Note that the in-degree polynomial serves as an effective invariant of virtual knots. 
\end{example}

\section{Questions} \label{Q}

We conclude with some questions for future research.

\begin{itemize}
\item What additional enhancements can be added to these quivers to further
strengthen the invariants?
\item What about the cases of pseudo-singular and virtual pseudo-singular knots,
which include classical, singular, pre- and virtual crossings? Including virtual
crossings by allowing new crossing types in the detour move does not change
the psyquandle structure, but interaction rules for precrossings with singular 
crossings may require additional axioms.
\item Adding operations at virtual crossings will likewise require a new
algebraic structure, but analogous quiver-valued invariants are surely possible.
\item What conditions on a psyquandle and set of endomorphisms lead to specific 
types of quivers? For instance, a constant endomorphism yields a star while an
automorphism produces a union of cycles. We are particularly interested in 
conditions yielding Hasse diagrams of posets.
\end{itemize}

\bibliography{jc-sn2020}{}
\bibliographystyle{abbrv}

\bigskip

\noindent
\textsc{Mathematics and Statistics Department \\
Hamilton  College \\
198 College Hill Rd. \\
Clinton, NY 13323} \\ \\
\textsc{Department of Mathematical Sciences \\
Claremont McKenna College \\
850 Columbia Ave. \\
Claremont, CA 91711}\\

\end{document}